\documentclass{amsart}

\usepackage{pgf,tikz,pgfplots}
\pgfplotsset{compat=1.15}
\usepackage{mathrsfs}
\usetikzlibrary{arrows}

\usepackage[T1]{fontenc}
\usepackage{microtype}
\usepackage{lmodern}
\usepackage[colorlinks=true,urlcolor=blue, citecolor=red,linkcolor=blue,linktocpage,pdfpagelabels, bookmarksnumbered,bookmarksopen]{hyperref}
\usepackage[hyperpageref]{backref}
\usepackage{amsthm} 
\usepackage{latexsym,amsmath,amssymb}

\usepackage{accents}
\usepackage{esint}

\usepackage{soul}
\usepackage{mathtools} 
\usepackage{xparse} 
\usepackage[capitalize]{cleveref}

\usepackage[shortlabels]{enumitem}

\usepackage[nomargin,inline,marginclue,draft]{fixme}
\FXRegisterAuthor{n}{nlang}{\color{red} Nicole}
\fxusetheme{color}
\usepackage{layout} 
\usepackage[top=3cm, bottom=3cm, left=2cm, right=2cm]{geometry} 

\usepackage[textsize=small]{todonotes}
\title[$n$-Laplace systems with antisymmetric potentials]{Non-Regularizing properties of ${n}$-Laplace systems with antisymmetric potentials in critical Lebesgue spaces}
\author{Dorian Martino}
 \address[Dorian Martino]{ETH Zürich, Department of Mathematics, Rämistrasse 101, 8092 Zürich, Switzerland}
 \email{dorian.martino@math.ethz.ch}
\author{Armin Schikorra}
 \address[Armin Schikorra]{Department of Mathematics,
 University of Pittsburgh,
 301 Thackeray Hall,
 Pittsburgh, PA 15260, USA}
 \email{armin@pitt.edu}

\newcommand{\B}{{\mathbb{B}}}

\renewcommand{\S}{{\mathbb S}}

\newtheorem{theorem}{Theorem}
\newtheorem{lemma}[theorem]{Lemma}

\newtheorem{proposition}[theorem]{Proposition}

\theoremstyle{definition}

\theoremstyle{remark}

\newcommand\curl{{\rm curl\,}}

\newcommand{\di}{\mathrm{div}}

\renewcommand{\so}{\mathfrak{so}}

\newcommand{\R}{\mathbb{R}}

\newcommand{\Nr}{\mathcal{N}}

\newcommand{\brac}[1]{\left (#1 \right )}
\newcommand{\abs}[1]{\left\lvert #1 \right \rvert}

\newcommand{\norm}[1]{\left\|{#1}\right\|}

\newcommand{\barint}{
\rule[.036in]{.12in}{.009in}\kern-.16in \displaystyle\int }

\newcommand{\barcal}{\text{$ \rule[.036in]{.11in}{.007in}\kern-.128in\int $}}

\def\mvint_#1{\mathchoice
          {\mathop{\vrule width 6pt height 3 pt depth -2.5pt
                  \kern -8pt \intop}\nolimits_{\kern -3pt #1}}%
          {\mathop{\vrule width 5pt height 3 pt depth -2.6pt
                  \kern -6pt \intop}\nolimits_{#1}}%
          {\mathop{\vrule width 5pt height 3 pt depth -2.6pt
                  \kern -6pt \intop}\nolimits_{#1}}%
          {\mathop{\vrule width 5pt height 3 pt depth -2.6pt
                  \kern -6pt \intop}\nolimits_{#1}}}

\numberwithin{theorem}{section} \numberwithin{equation}{section}

\newcommand{\lap}{\Delta }
\newcommand{\aleq}{\lesssim}

\newcommand{\aeq}{\approx}

\newcommand{\dd}{\mathrm{d}}

\usepackage{scalerel}[2014/03/10]
\usepackage[usestackEOL]{stackengine}
\def\avint{\,\ThisStyle{\ensurestackMath{%
			\stackinset{c}{.2\LMpt}{c}{.5\LMpt}{\SavedStyle-}{\SavedStyle\phantom{\int}}}%
		\setbox0=\hbox{$\SavedStyle\int\,$}\kern-\wd0}\int}

\renewcommand{\div}{\operatorname{div}}

\let\latexchi\chi
\makeatletter
\renewcommand\chi{\@ifnextchar_\sub@chi\latexchi}
\newcommand{\sub@chi}[2]{
  \@ifnextchar^{\subsup@chi{#2}}{\latexchi^{}_{#2}}%
}
\newcommand{\subsup@chi}[3]{
  \latexchi_{#1}^{#3}%
}
\makeatother
\newcommand{\eps}{\varepsilon}

\date{\today}

\begin{document}

\begin{abstract}
For every $n\geq 3$ we construct a bounded discontinuous map $u\in W^{1,n}(\B^n,\R^{n+2})$ solving an $n$-Laplace system with an antisymmetric potential $\Omega\in L^n$. This shows that a recent regularity result on $n$-Laplace systems with antisymmetric potentials in Lorentz spaces by the authors \cite{MartinoSchikorra} is sharp in the sense that we cannot move from Lorentz spaces to classical Lebesgue spaces.
This gives in particular a negative answer to Rivi\`ere's question in \cite[p.~107, equation~(III.23)]{RiviereSurvey}. 

\end{abstract}

\maketitle
\tableofcontents

\section{Introduction}
Let $\Nr\subset \R^N$ be a smooth closed submanifold and let $\B^n\subset \R^n$ be the unit ball centred at the origin. A critical point $u\in W^{1,n}(\B^n,\R^N)$ of the following $n$-energy is called an $n$-harmonic map
\[
 \forall u\in W^{1,n}(\B^n,\Nr),\qquad E(u) \coloneqq \int_{\B^n} |\nabla u|^{n}.
\]
Rivière \cite{R07} proved that the Euler--Lagrange system can be written in the form
\begin{equation}\label{eq:system}
 \forall i\in\{1,\ldots,N\},\qquad -\di\bigl(\abs{\nabla u}^{n-2}\nabla u^i\bigr) =\sum_{j=1}^N \Omega_{ij}\cdot \abs{\nabla u}^{n-2}\nabla u^j,
\end{equation}
where $\Omega_{ij} \in L^n(\B^n,\R^n)$ depends on $u$ and $\Nr$ and is antisymmetric $\Omega_{ij}=-\Omega_{ji}$, owing to the fact the second fundamental form is perpendicular to the tangent space. For harmonic maps (i.e.\ the case $n=2$), Rivière proved that, even if $\Omega$ does not depend on $u$, such solutions are continuous: the algebraic property of antisymmetry of $\Omega$ leads to compensation $\di$-$\curl$ products, via Uhlenbeck's gauge theory \cite{U82}, and thus to regularity.\\

For $n \geq 3$ the situation changes drastically. Regularity for critical $n$-harmonic maps into manifolds $\Nr$ is only known under some additional assumption on geometric properties of $\Nr$ (such that $\Nr$ being a round sphere), on $u$ (such as minimizing, or arising from a min-max scheme) or some additional regularity assumption on $u$ \cite{DL2026,FR2013,F93,GX2019,MMS2026,MartinoSchikorra,MPS22,MY96,SS17,S94,takeuchi1994,TW1995,U1977}.
In \cite[p.~107, Eq.~(III.23)]{RiviereSurvey} Rivi\`ere asked whether every $W^{1,n}$ solution of \eqref{eq:system} is continuous under the sole assumptions that $\Omega$ is antisymmetric and belongs to $L^n$.
Towards this, the following was proven \cite[Theorem~1.2]{MartinoSchikorra}.
\begin{theorem}\label{thm:MS}
Let $n\geq3$, let $\B^n\subset\R^n$ be the unit ball centred at the origin, and let $u\in W^{1,n}(\B^n,\R^N)$ be a solution to \eqref{eq:system}.  Assume that 
\begin{equation}\label{eq:MS-Lorentz}
 \Omega_{ij}=-\Omega_{ji}, \qquad \text{ and }\qquad \Omega\in L^{(n,2)}\bigl(\B^n,\so(N)\otimes\R^n\bigr),
\end{equation}
and, with $R_\alpha=\partial_\alpha(-\Delta)^{-1/2}$,
\begin{equation}\label{eq:MS-curl}
 \max_{1\leq i,j\leq N} \max_{1\leq\alpha,\beta\leq n} \norm{R_\alpha\Omega_{ij}^{\beta}-R_\beta\Omega_{ij}^{\alpha}}_{L^{(n,1)}(\B^n)} <\infty.
\end{equation}
Then $u$ is continuous in $\B^n$.
\end{theorem}

\Cref{thm:MS} followed a series of partial results that assumed higher regularity differentiability on $\Omega$, in particular \cite{MPS22,SS17}. It merges Uhlenbeck's gauge theoretic arguments with the sharp potential estimates for the vectorial $p$-Lapalcian by Kuusi-Mingione \cite{KM18}, and thus one can wonder how sharp it is. Given the progress in Artificial Intelligence, we do not even have to work that hard to show that one cannot hope to remove the Lorentz space assumption $\Omega \in L^{(n,2)}(\B^n,\so(N)\otimes \R^n)$, a simple prompt obtains the following:

\begin{theorem}\label{thm:Ln-counterexample}
For every $n\geq3$, with $N=n+2$, there are explicit
\begin{equation}\label{eq:thm1-spaces}
\begin{aligned}
 u&\in W^{1,n}(\B^n,\R^N)\cap L^\infty(\B^n,\R^N),\\
 \Omega&\in L^n\bigl(\B^n,\so(N)\otimes\R^n\bigr)
\end{aligned}
\end{equation}
such that \eqref{eq:system} holds in $\mathcal D'(\B^n)$ and $u$ is discontinuous at the origin.  Moreover, we have
\begin{equation}\label{eq:thm1-Omega-size}
 \abs{\nabla u(x)} \aeq \abs{\Omega(x)}\simeq_n
 \frac{1}{\abs{x}\sqrt{\log(e/\abs{x})}},
\end{equation}
and consequently, for every finite $q>0$,
\begin{equation}\label{eq:thm1-full-Lorentz-range}
 \nabla u,\, \Omega\in L^{(n,q)}(\B^n)
 \quad\Longleftrightarrow\quad
 q>2.
\end{equation}
\end{theorem}

We can ensure the curl assumption \eqref{eq:MS-curl} is trivially satisfied, but we need dimension $n \geq 4$.

\begin{theorem}\label{thm:gradient-counterexample}
For every $n\geq4$, with $N=n+2$, there are explicit
\begin{equation}\label{eq:thm2-spaces}
 u\in W^{1,n}(\B^n,\R^N)\cap L^\infty(\B^n,\R^N), \qquad \Xi\in W^{1,n}(\B^n,\so(N)),
\end{equation}
with $\Omega\coloneqq \nabla\Xi$ satisfying
\begin{equation}\label{eq:thm2-Omega-spaces}
 \Omega\in L^n\bigl(\B^n,\so(N)\otimes\R^n\bigr),
\end{equation}
such that \eqref{eq:system} holds in $\mathcal D'(\B^n)$, and $u$ is discontinuous at $0$.
\end{theorem}

The above theorems give in particular a negative answer to Rivi\`ere's question in \cite{RiviereSurvey}, antisymmetry alone does not regularize the equation \eqref{eq:system}.

\subsection*{Usage of AI} 
Ideas and the first draft of this article was contributed by ChatGPT. Mathematical validation, final proof decisions, and final wording remain the sole responsibility of the human author, but no (deeper) ideas were provided by the authors.

\subsection*{Acknowledgement}
A.S. is funded by NSF Career DMS-2044898. D.M. is funded by Swiss National Science Foundation, project SNF 200020\textunderscore 219429.

\section{The main computation}

We will need the following result that can be found in the proof of Lemma 4.1 in \cite{MS2024}.
\begin{lemma}\label{lem:Lorentz-profile}
For $0<a<1$, set
\begin{equation}\label{eq:Fa}
 F_a(x)\coloneqq \frac{1}{\abs{x}\bigl(\log(e/\abs{x})\bigr)^a},
 \qquad x\in\B^n\setminus\{0\}.
\end{equation}
For any $0<\sigma<\infty$, we have $F_a\in L^{(n,\sigma)}(\B^n)$
 if and only if $a\sigma>1$. In particular, $F_a\in L^n$ if and only if $an>1$.
\end{lemma}
%
%
%
%

Throughout the remainder of the paper for $x\in\B^n\setminus\{0\}$ we use the polar variables
\begin{equation}\label{eq:polar-variables}
 r\coloneqq \abs{x}, \qquad \theta \coloneqq \frac{x}{|x|}, \qquad t \coloneqq \log\frac {e}{|x|} = \log \frac{e}{r}\geq 1.
\end{equation}

\begin{proposition}\label{lem:common-computation}
Let $f,g,\psi\in C^\infty([1,\infty))$, let $\alpha,\gamma,\kappa\in C^\infty([1,\infty))$, and define $u\colon \B^n\setminus\{0\}\to\R^{n+2}$ by
\begin{equation}\label{eq:general-u}
 u(x)\coloneqq 
 \begin{pmatrix}
  \cos\psi(t)\\[1mm]
  \sin\psi(t)\\[1mm]
  g(t)\dfrac{x}{\abs{x}}
 \end{pmatrix}
 \in\R^2\times\R^n.
\end{equation}

Assume
\begin{itemize}
 \item for every $t\geq 1$, it holds
 \begin{equation}\label{eq:general-positive-profiles}
  f(t)>0,
  \qquad
  g(t)>0;
 \end{equation}
 \item for every $t\geq 1$, it holds
 \begin{equation}\label{eq:general-phase-relation}
  \partial_t\psi(t)=f(t).
 \end{equation}
\item for every $t\geq 1$, it holds
\begin{equation}\label{eq:general-coefficient-relation}
 \kappa(t)f(t)
 +(n-1)g(t)\alpha(t)
 -\partial_tg(t)\gamma(t)
 =f(t)^2.
\end{equation}
\end{itemize}

For $1\leq i,j\leq n+2$, set
\begin{equation}\label{eq:general-Omega}
 \R^n \ni \Omega_{ij}(x):=
 \begin{cases}
 \displaystyle -\frac{\kappa(t)}r\theta, & i=1,\quad j=2, \\[2mm]
 \displaystyle \frac{\kappa(t)}r\theta, & i=2,\quad j=1, \\[2mm]
 \displaystyle \brac{ \alpha(t)\cos\psi(t)-\beta(t)\sin\psi(t) } \nabla\brac{\frac{x_k}{|x|}} \\[3mm]
 \displaystyle\qquad +\frac{1}{|x|} \brac{ \gamma(t)\cos\psi(t)-\delta(t)\sin\psi(t) } \frac{x_k}{|x|}\theta, & i=1,\quad j=k+2,\quad 1\leq k\leq n, \\[3mm]
 \displaystyle \brac{ \alpha(t)\sin\psi(t)+\beta(t)\cos\psi(t) } \nabla\brac{\frac{x_k}{|x|}} \\[3mm]
 \displaystyle\qquad +\frac{1}{|x|} \brac{ \gamma(t)\sin\psi(t)+\delta(t)\cos\psi(t) } \frac{x_k}{|x|}\theta, & i=2,\quad j=k+2,\quad 1\leq k\leq n, \\[3mm]
 -\Omega_{ji}(x), & i=k+2,\quad j\in\{1,2\},\quad 1\leq k\leq n, \\[2mm]
 0, & i=j\quad\text{or}\quad 3\leq i,j\leq n+2.
 \end{cases}
\end{equation}

where
$\delta,\beta:[1,\infty)\to\R$ are defined by
\begin{equation}\label{eq:general-delta}
\begin{split}
 \delta(t) \coloneqq \frac1{f(t)}
 \left[
 (n-1)g(t)  -
 \frac{
 \partial_t\brac{
  \brac{
   f(t)^2
   +\brac{\partial_tg(t)}^2
   +(n-1)g(t)^2
  }^{\frac{n-2}{2}}
  \partial_tg(t)
 }
 }{
  \brac{
   f(t)^2
   +\brac{\partial_tg(t)}^2
   +(n-1)g(t)^2
  }^{\frac{n-2}{2}}
 }
 \right]
\end{split}
\end{equation}
and
\begin{equation}\label{eq:general-beta}
\begin{split}
 \beta(t) \coloneqq \frac1{(n-1)g(t)}
 \left[
 -
 \frac{
 \partial_t\brac{
  \brac{
   f(t)^2
   +\brac{\partial_tg(t)}^2
   +(n-1)g(t)^2
  }^{\frac{n-2}{2}}
  f(t)
 }
 }{
  \brac{
   f(t)^2
   +\brac{\partial_tg(t)}^2
   +(n-1)g(t)^2
  }^{\frac{n-2}{2}}
 } 
 +\partial_tg(t)\delta(t)
 \right].
\end{split}
\end{equation}

Equivalently, \eqref{eq:general-Omega} can be displayed as
\begin{equation}\label{eq:general-Omega-matrix}
\resizebox{\textwidth}{!}{$\displaystyle
\Omega(x)
=
\left(
\begin{array}{ccccc}
0
&
-\dfrac{\kappa(t)}{|x|}\theta
&
\begin{aligned}
&
\brac{
 \alpha(t)\cos\psi(t)-\beta(t)\sin\psi(t)
}
\nabla\brac{\frac{x_1}{|x|}}
\\
&\qquad
+
\frac{1}{|x|}
\brac{
 \gamma(t)\cos\psi(t)-\delta(t)\sin\psi(t)
}
\frac{x_1}{|x|}\theta
\end{aligned}
&
\ldots
&
\begin{aligned}
&
\brac{
 \alpha(t)\cos\psi(t)-\beta(t)\sin\psi(t)
}
\nabla\brac{\frac{x_n}{|x|}}
\\
&\qquad
+
\frac{1}{|x|}
\brac{
 \gamma(t)\cos\psi(t)-\delta(t)\sin\psi(t)
}
\frac{x_n}{|x|}\theta
\end{aligned}
\\[4ex]
\dfrac{\kappa(t)}{|x|}\theta
&
0
&
\begin{aligned}
&
\brac{
 \alpha(t)\sin\psi(t)+\beta(t)\cos\psi(t)
}
\nabla\brac{\frac{x_1}{|x|}}
\\
&\qquad
+
\frac{1}{|x|}
\brac{
 \gamma(t)\sin\psi(t)+\delta(t)\cos\psi(t)
}
\frac{x_1}{|x|}\theta
\end{aligned}
&
\ldots
&
\begin{aligned}
&
\brac{
 \alpha(t)\sin\psi(t)+\beta(t)\cos\psi(t)
}
\nabla\brac{\frac{x_n}{|x|}}
\\
&\qquad
+
\frac{1}{|x|}
\brac{
 \gamma(t)\sin\psi(t)+\delta(t)\cos\psi(t)
}
\frac{x_n}{|x|}\theta
\end{aligned}
\\[4ex]
\begin{aligned}
&
-\brac{
 \alpha(t)\cos\psi(t)-\beta(t)\sin\psi(t)
}
\nabla\brac{\frac{x_1}{|x|}}
\\
&\qquad
-
\frac{1}{|x|}
\brac{
 \gamma(t)\cos\psi(t)-\delta(t)\sin\psi(t)
}
\frac{x_1}{|x|}\theta
\end{aligned}
&
\begin{aligned}
&
-\brac{
 \alpha(t)\sin\psi(t)+\beta(t)\cos\psi(t)
}
\nabla\brac{\frac{x_1}{|x|}}
\\
&\qquad
-
\frac{1}{|x|}
\brac{
 \gamma(t)\sin\psi(t)+\delta(t)\cos\psi(t)
}
\frac{x_1}{|x|}\theta
\end{aligned}
&
0
&
\ldots
&
0
\\[4ex]
\vdots
&
\vdots
&
\vdots
&
\ddots
&
\vdots
\\[4ex]
\begin{aligned}
&
-\brac{
 \alpha(t)\cos\psi(t)-\beta(t)\sin\psi(t)
}
\nabla\brac{\frac{x_n}{|x|}}
\\
&\qquad
-
\frac{1}{|x|}
\brac{
 \gamma(t)\cos\psi(t)-\delta(t)\sin\psi(t)
}
\frac{x_n}{|x|}\theta
\end{aligned}
&
\begin{aligned}
&
-\brac{
 \alpha(t)\sin\psi(t)+\beta(t)\cos\psi(t)
}
\nabla\brac{\frac{x_n}{|x|}}
\\
&\qquad
-
\frac{1}{|x|}
\brac{
 \gamma(t)\sin\psi(t)+\delta(t)\cos\psi(t)
}
\frac{x_n}{|x|}\theta
\end{aligned}
&
0
&
\ldots
&
0
\end{array}
\right).
$}
\end{equation}

Then

\begin{itemize} \item $\Omega_{ij}=-\Omega_{ji}$
\item Moreover,
\begin{align}
 \abs{u(x)}^2
 &=1+g(t)^2,
 \label{eq:general-u-size}\\
 \abs{\nabla u(x)}^2
 &=\frac1{r^2}
 \brac{
  f(t)^2
  +\brac{\partial_tg(t)}^2
  +(n-1)g(t)^2
 }.
 \label{eq:general-gradient-size}
\end{align}

\item we have for every $i\in\{1,\dots,n+2\}$
\begin{equation}\label{eq:general-system}
 -\div\brac{\abs{\nabla u}^{n-2}\nabla u^i}
 =
 \sum_{j=1}^{n+2}
 \Omega_{ij}\cdot
 \abs{\nabla u}^{n-2}\nabla u^j
 \qquad
 \text{in }\B^n\setminus\{0\},
\end{equation}

\item Furthermore,
\begin{equation}\label{eq:general-Omega-size}
\begin{split}
 \abs{\Omega(x)}^2 =
 \sum_{i,j=1}^{n+2}\abs{\Omega_{ij}(x)}^2 =
 \frac2{r^2}
 \brac{
  \kappa(t)^2
  +(n-1)\brac{\alpha(t)^2+\beta(t)^2}
  +\gamma(t)^2
  +\delta(t)^2
 }.
\end{split}
\end{equation}
\item Finally,
\begin{equation}\label{eq:energy-identity}
 \int_{\B^n}\abs{\nabla u}^n\,\dd x
 =
 \abs{\S^{n-1}}
 \int_1^\infty
 \brac{
  f(t)^2
  +\brac{\partial_tg(t)}^2
  +(n-1)g(t)^2
 }^{\frac n2}
 \,\dd t.
\end{equation}
\end{itemize}
\end{proposition}

\begin{proof}
We start by observing that the antisymmetry $\Omega_{ij}=-\Omega_{ji}$ is obvious. So is \eqref{eq:general-u-size}.

A few standard computations and observations:
for $1\leq k\leq n$, direct differentiation gives
\begin{equation}\label{eq:general-polar-identities}
 \nabla t=-\frac{\theta}{r},
 \qquad
 \theta\cdot\nabla\brac{\frac{x_k}{|x|}}=0,
 \qquad
 \Delta\brac{\frac{x_k}{|x|}}
 =-\frac{n-1}{r^2}\frac{x_k}{|x|}, \qquad \div (\theta) = \frac{n-1}{r}
\end{equation}
Moreover,
\begin{equation}\label{eq:general-spherical-identities}
 \sum_{k=1}^n\frac{x_k^2}{r^2}=1,
 \qquad
 \sum_{k=1}^n
 \abs{\nabla\brac{\frac{x_k}{|x|}}}^2
 =\frac{n-1}{r^2}.
\end{equation}

For the computation, set
\begin{equation}\label{eq:J}
 J\coloneqq 
 \begin{pmatrix}
  0&-1\\
  1&0
 \end{pmatrix}.
\end{equation}

We define
\begin{equation}\label{eq:general-w-v}
 w(t)\coloneqq 
 \begin{pmatrix}
  \cos\psi(t)\\
  \sin\psi(t)
 \end{pmatrix},
 \qquad
 v(t):=Jw(t)=
 \begin{pmatrix}
  -\sin\psi(t)\\
  \cos\psi(t)
 \end{pmatrix}.
\end{equation}

By chain rule and \eqref{eq:general-phase-relation},
\begin{equation}\label{eq:general-w-v-derivatives}
 \partial_tw(t)=f(t)v(t),
 \qquad
 \partial_tv(t)=-f(t)w(t).
\end{equation}

It follows from \eqref{eq:general-w-v-derivatives} and
\eqref{eq:general-polar-identities} that
\begin{equation}\label{eq:general-derivative-first}
 \partial_{x_k} w(x)=-\frac{f(t)}{|x|} v(t)\, \frac{x_k}{|x|} = - \frac{1}{r} f(t) v(t) \theta_k.
\end{equation}
In particular, we have
\[
 |\nabla w|^2 = \frac{|f(t)|^2}{r^2}.
\]

Similarly, for $1\leq \ell\leq n$,
\begin{equation}\label{eq:general-derivative-last}
 \nabla\brac{g(t)\frac{x_\ell}{|x|}}
 =g(t)\nabla\brac{\frac{x_\ell}{|x|}}
 -\frac{\partial_tg(t)}{r}\frac{x_\ell}{|x|}\theta.
\end{equation}
The two terms on the right-hand side of
\eqref{eq:general-derivative-last} are orthogonal, thus
\[
\begin{split}
 \sum_{\ell=1}^n\abs{\nabla\brac{g(t)\frac{x_\ell}{|x|}}}^2 =& g(t)^2 \sum_{\ell=1}^n \abs{\nabla\brac{\frac{x_\ell}{|x|}}}^2 + \sum_{\ell=1}^n\abs{\frac{\partial_tg(t)}{r}\frac{x_\ell}{|x|}}^2\\[2mm]
 \overset{\eqref{eq:general-spherical-identities}}{=}& g(t)^2 \frac{n-1}{r^2} + \frac{\abs{\partial_tg(t)}^2}{r^2}.
\end{split}
 \]
From \eqref{eq:general-u} we thus find \eqref{eq:general-gradient-size}
\[
  \abs{\nabla u(x)}^2
=\frac1{r^2}
 \brac{
  f(t)^2
  +\brac{\partial_tg(t)}^2
  +(n-1)g(t)^2
 }.
\]

For the remainder of the proof, we set
\begin{equation}\label{eq:general-mu}
 \mu(t)\coloneqq
 \brac{
  f(t)^2
  +\brac{\partial_tg(t)}^2
  +(n-1)g(t)^2
 }^{\frac{n-2}{2}}.
\end{equation}
Then, we have
\begin{equation}\label{eq:general-weight}
 \abs{\nabla u(x)}^{n-2}=r^{2-n}\mu(t).
\end{equation}
We have $\mu>0$ by \eqref{eq:general-positive-profiles} and $\delta,\beta:[1,\infty)\to\R$ from \eqref{eq:general-delta} and \eqref{eq:general-beta} become
\begin{equation}\label{eq:general-deltav2}
\begin{split}
 \delta(t)
 &\coloneqq \frac1{f(t)}
 \Bigg[
 (n-1)g(t)
  -
 \frac{
 \partial_t\brac{
  \mu(t)
  \partial_tg(t)
 }
 }{\mu(t)}
 \Bigg],
\end{split}
\end{equation}
and
\begin{equation}\label{eq:general-betav2}
\begin{split}
 \beta(t)
 &\coloneqq \frac1{(n-1)g(t)}
 \Bigg[
 -
 \frac{
 \partial_t\brac{
  \mu(t)
  f(t)
 }
 }{\mu(t)}
 +\partial_tg(t)\delta(t)
 \Bigg].
\end{split}
\end{equation}

To compute $\div(|\nabla u|^{n-2} \nabla u)$ we first observe that for every scalar or vector-valued function $a=a(t)$,
\begin{equation}\label{eq:general-radial-divergence}
 \div\brac{r^{1-n}a(t)\theta} =-r^{-n}\, \partial_ta(t).
\end{equation}
Indeed, this follows from a direct computation using \eqref{eq:general-polar-identities} for the second and the fourth equality
\[
\begin{split}
 \div\brac{r^{1-n}a(t)\theta} & = r^{1-n} a(t) \div(\theta) + \theta \cdot \nabla \brac{ r^{1-n} a(t)}\\[2mm]
 & =  r^{1-n} a(t) \frac{n-1}{r} + \theta \cdot \nabla \brac{ r^{1-n} a(t)}\\[2mm]
 & =(n-1)r^{-n} a(t)  + (1-n)  r^{-n}  a(t) + r^{1-n}  \theta \cdot  \nabla a(t)\\[2mm]
 & =-r^{1-n} \, \theta \cdot  \partial_t a(t)\, \frac{\theta}{r} .
 \end{split}
\]

Using \eqref{eq:general-weight}, we obtain  
\begin{equation}\label{eq:general-left-first}
\begin{split}
 -\div\brac{\abs{\nabla u}^{n-2}\nabla w}
  &\overset{\eqref{eq:general-derivative-first}}{=} -\div\brac{-r^{1-n}\mu(t) f(t) v(t) \theta} \\[2mm]
  &\overset{\eqref{eq:general-radial-divergence}}{=} -r^{-n}\partial_t \brac{\mu(t) f(t) v(t) } \\[2mm]
 &\overset{\eqref{eq:general-w-v-derivatives}}{=}r^{-n} \brac{ \mu(t) f(t)^2 w(t) -\partial_t\brac{\mu(t) f(t)}v(t) }\\[2mm]
 & \overset{\eqref{eq:general-weight}}{=} |\nabla u|^{n-2} r^{-2} \brac{ f(t)^2 w(t) -\frac{\partial_t\brac{\mu(t) f(t)}v(t)}{\mu(t)} }.
\end{split}
\end{equation}

For the last $n$ components,
\eqref{eq:general-derivative-last} and
\eqref{eq:general-weight} give
\begin{equation}\label{eq:general-weighted-last-gradient}
\begin{split}
 \abs{\nabla u}^{n-2}
 \nabla\brac{g(t) \frac{x_k}{|x|}}
 =
 r^{2-n}\mu(t) g(t)
 \nabla\brac{\frac{x_k}{|x|}} -
 r^{1-n}\mu(t)\,\partial_tg(t)
 \frac{x_k}{|x|}\theta.
\end{split}
\end{equation}
The coefficient $r^{2-n}\mu(t) g(t)$ is radial, thus $\nabla \brac{r^{2-n}\mu(t) g(t)} \cdot 
\nabla\brac{x_k/|x|} = 0$ by \eqref{eq:general-polar-identities}, so
\[
 \div\brac{r^{2-n}\mu(t) g(t)
 \nabla\brac{\frac{x_k}{|x|}}
 }={r^{2-n}\mu(t) g(t)
 \lap \brac{\frac{x_k}{|x|}}
 }
 \overset{\eqref{eq:general-polar-identities}}{=}-(n-1)r^{-n}\mu(t) g(t)\frac{x_k}{|x|}.
\]
Similarly, using
$\theta\cdot\nabla\brac{x_k/|x|}=0$ and
\eqref{eq:general-radial-divergence},
\[
 \div\brac{
 -
 r^{1-n}\mu(t)\,\partial_tg(t)
 \frac{x_k}{|x|}\theta
 } = \div\brac{
 -
 r^{1-n}\mu(t)\,\partial_tg(t)
 \theta
 } \frac{x_k}{|x|}
 =r^{-n}
 \partial_t\brac{\mu(t)\,\partial_tg(t)}
 \frac{x_k}{|x|}.
\]
Therefore, again with \eqref{eq:general-weight},
\begin{equation}\label{eq:general-left-last}
\begin{split}
 -\div\brac{
  \abs{\nabla u}^{n-2}
  \nabla\brac{g(t)\frac{x_k}{|x|}}
 }
 =|\nabla u|^{n-2} r^{-2}
 \brac{
  (n-1) g(t)
  -\frac{\partial_t\brac{\mu(t)\,\partial_tg(t)}} {\mu(t)}
 } 
 \frac{x_k}{|x|}. 
\end{split}
\end{equation}

Now we compute $\Omega \cdot \nabla u$. We begin by computing $\sum_{j=1}^{n+2} \Omega_{ij} \cdot \nabla u^j$ for $i=1,2$, and for this we first observe
\begin{equation}\label{eq:Omega12nablau12}
 \begin{pmatrix}
  \Omega_{12}\cdot\nabla u^2\\
  \Omega_{21}\cdot\nabla u^1
 \end{pmatrix}
 = \begin{pmatrix}
  \Omega_{12}\cdot\nabla w^2\\
  \Omega_{21}\cdot\nabla w^1
 \end{pmatrix} =r^{-2}\kappa(t)f(t)\,w(t),
\end{equation}
Indeed, by \eqref{eq:general-derivative-first}
\[
 \nabla u^1=-\frac{f(t)}{r}v_1(t)\theta,
 \qquad
 \nabla u^2=-\frac{f(t)}{r}v_2(t)\theta.
\]
Hence, using $\Omega_{12}=-\frac{\kappa(t)}{r}\theta$, $\Omega_{21}=\frac{\kappa(t)}{r}\theta$, and $\abs{\theta}^2=1$,
\[
 \Omega_{12}\cdot\nabla u^2
 =\frac{\kappa(t)f(t)}{r^2}v_2(t),
 \qquad
 \Omega_{21}\cdot\nabla u^1
 =-\frac{\kappa(t)f(t)}{r^2}v_1(t).
\]
By \eqref{eq:general-w-v}, $v_1(t)=-w_2(t)$ and $v_2(t)=w_1(t)$, so we obtain \eqref{eq:Omega12nablau12}.

Next we compute the mixed contribution. We claim
\begin{equation}\label{eq:general-right-first-mixed}
 \begin{pmatrix}
  \displaystyle\sum_{k=1}^n\Omega_{1,k+2}\cdot\nabla\brac{g(t)\frac{x_k}{|x|}}\\[1ex]
  \displaystyle\sum_{k=1}^n\Omega_{2,k+2}\cdot\nabla\brac{g(t)\frac{x_k}{|x|}}
 \end{pmatrix}
 =\frac1{r^2}
 \Bigl[
  (n-1)g(t)\bigl(\alpha(t)w(t)+\beta(t)v(t)\bigr)
  -\partial_tg(t)\bigl(\gamma(t)w(t)+\delta(t)v(t)\bigr)
 \Bigr].
\end{equation}
Indeed, by \eqref{eq:general-Omega} and \eqref{eq:general-w-v}, for $i=1,2$ and $1\leq k\leq n$,
\begin{equation}\label{eq:Omegaikp2}
 \Omega_{i,k+2}
 =\bigl(\alpha(t)w(t)+\beta(t)v(t)\bigr)_i\nabla\brac{\frac{x_k}{|x|}}
 +\frac1r\bigl(\gamma(t)w(t)+\delta(t)v(t)\bigr)_i\frac{x_k}{|x|}\theta.
\end{equation}
 Dotting with \eqref{eq:general-derivative-last}
 and using $\theta\cdot\nabla\brac{x_k/|x|}=0$ from \eqref{eq:general-polar-identities},
\[
 \Omega_{i,k+2}\cdot\nabla\brac{g(t)\frac{x_k}{|x|}}
 =\bigl(\alpha(t)w(t)+\beta(t)v(t)\bigr)_i\,g(t)\abs{\nabla\brac{\frac{x_k}{|x|}}}^2
 -\bigl(\gamma(t)w(t)+\delta(t)v(t)\bigr)_i\,\partial_tg(t)\,\frac{x_k^2}{r^4}.
\]
Summing over $k$ and using \eqref{eq:general-spherical-identities},
\[
 \sum_{k=1}^n\Omega_{i,k+2}\cdot\nabla\brac{g(t)\frac{x_k}{|x|}}
 =\frac1{r^2}
 \Bigl[
  (n-1)g(t)\bigl(\alpha(t)w(t)+\beta(t)v(t)\bigr)_i
  -\partial_tg(t)\bigl(\gamma(t)w(t)+\delta(t)v(t)\bigr)_i
 \Bigr],
\]
which is \eqref{eq:general-right-first-mixed}.

Thus for $i=1,2$ we have
\[
\begin{split}
 \sum_{j=1}^{n+2} \Omega_{ij} \cdot \nabla u^j
 &=r^{-2}\kappa(t)f(t)\,w_i(t) +\frac1{r^2} \Bigl[ (n-1)g(t)\bigl(\alpha(t)w(t)+\beta(t)v(t)\bigr)_i -\partial_tg(t)\bigl(\gamma(t)w(t)+\delta(t)v(t)\bigr)_i \Bigr]\\[2mm]
 &=r^{-2} \Bigl[ \bigl(\kappa(t)f(t)+(n-1)g(t)\alpha(t)-\partial_tg(t)\gamma(t)\bigr)w_i(t) +\bigl((n-1)g(t)\beta(t)-\partial_tg(t)\delta(t)\bigr)v_i(t) \Bigr]\\[2mm]
 &\overset{\eqref{eq:general-coefficient-relation}}{=}
 r^{-2} \Bigl[ f(t)^2w_i(t) +\bigl((n-1)g(t)\beta(t)-\partial_tg(t)\delta(t)\bigr)v_i(t) \Bigr]\\[2mm]
 &\overset{\eqref{eq:general-betav2}}{=}
 r^{-2} \Bigl[ f(t)^2w_i(t) -\frac{\partial_t\brac{\mu(t)f(t)}}{\mu(t)}v_i(t) \Bigr].
\end{split}
\]

Comparing with \eqref{eq:general-left-first}, we conclude
\[
 \abs{\nabla u}^{n-2}\sum_{j=1}^{n+2} \Omega_{ij}\cdot\nabla u^j
 =-\div\brac{\abs{\nabla u}^{n-2}\nabla u^i}
 \qquad
 \text{for }i=1,2,
\]
which is \eqref{eq:general-system} for $i=1,2$.

For $i=k+2$ with $1\leq k\leq n$, we argue analogously. Since $\Omega_{k+2,j}=0$ for $3\leq j\leq n+2$ and $\Omega_{k+2,j}=-\Omega_{j,k+2}$ for $j=1,2$, we have
\[
 \sum_{j=1}^{n+2}\Omega_{k+2,j}\cdot\nabla u^j
 =-\sum_{i=1}^2\Omega_{i,k+2}\cdot\nabla u^i.
\]

By \eqref{eq:general-derivative-first}, $\nabla u^i=-\frac{f(t)}r v_i(t)\theta$ for $i=1,2$, and by \eqref{eq:Omegaikp2},
again since $\theta\cdot\nabla\brac{x_k/|x|}=0$ by \eqref{eq:general-polar-identities}, only the $\theta$-term survives upon dotting with $\nabla u^i$:
\[
 \Omega_{i,k+2}\cdot\nabla u^i
 =-\frac{f(t)}{r^2}\frac{x_k}{|x|}\,v_i(t)\bigl(\gamma(t)w(t)+\delta(t)v(t)\bigr)_i.
\]
Summing over $i=1,2$ and using that $w(t),v(t)$ are orthonormal, so $w(t)\cdot v(t)=0$ and $\abs{v(t)}^2=1$,
\[
 \sum_{i=1}^2 v_i(t)\bigl(\gamma(t)w(t)+\delta(t)v(t)\bigr)_i
 =\gamma(t)\,w(t)\cdot v(t)+\delta(t)\abs{v(t)}^2
 =\delta(t).
\]
Hence
\begin{equation}\label{eq:general-right-last}
 \sum_{j=1}^{n+2}\Omega_{k+2,j}\cdot\nabla u^j
 =\frac1{r^2}f(t)\delta(t)\,\frac{x_k}{|x|},
 \qquad
 1\leq k\leq n.
\end{equation}

By \eqref{eq:general-deltav2},
we have \eqref{eq:general-right-last} reads
\[
 \sum_{j=1}^{n+2}\Omega_{k+2,j}\cdot\nabla u^j
 =\frac1{r^2}
 \brac{
  (n-1)g(t)
  -\frac{\partial_t\brac{\mu(t)\partial_tg(t)}}{\mu(t)}
 }
 \frac{x_k}{|x|}.
\]
Comparing with \eqref{eq:general-left-last}, we conclude
\[
 \abs{\nabla u}^{n-2}\sum_{j=1}^{n+2}\Omega_{k+2,j}\cdot\nabla u^j
 =-\div\brac{\abs{\nabla u}^{n-2}\nabla u^{k+2}},
 \qquad
 1\leq k\leq n,
\]
so \eqref{eq:general-system} also holds for $i=k+2$, $1\leq k\leq n$, i.e. for $i=3,\dots,n+2$. Together with the case $i=1,2$ established above, this proves \eqref{eq:general-system} for every $i=1,\dots,n+2$.

Next, to prove \eqref{eq:general-Omega-size} we observe that since $\theta\cdot\nabla\brac{x_k/|x|}=0$ by \eqref{eq:general-polar-identities},
\[
\begin{split}
  &\abs{\brac{
  \alpha(t)\cos\psi(t)-\beta(t)\sin\psi(t)
 }
 \nabla\brac{\frac{x_k}{|x|}}
 +\frac{1}{|x|}
 \brac{
  \gamma(t)\cos\psi(t)-\delta(t)\sin\psi(t)
 }
 \frac{x_k}{|x|}\theta}^2\\[2mm]
 = &\abs{\brac{
  \alpha(t)\cos\psi(t)-\beta(t)\sin\psi(t)
 }
 \nabla\brac{\frac{x_k}{|x|}}}^2
 +\frac{1}{r^2} \abs{
 \brac{
  \gamma(t)\cos\psi(t)-\delta(t)\sin\psi(t)
 }
 \frac{x_k}{|x|}}^2\\[2mm]
 =&\abs{\brac{
  \alpha(t)w(t)+\beta(t) v(t) }_1
 \nabla\brac{\frac{x_k}{|x|}}}^2
 +\frac{1}{r^2} \abs{
 \brac{
  \gamma(t)w(t)+\delta(t)v(t)
 }_1
 \frac{x_k}{|x|}}^2
 \end{split}
\]
and similarly,
\[
\begin{split}
&
\abs{\brac{ 
 \alpha(t)\sin\psi(t)+\beta(t)\cos\psi(t)
}
\nabla\brac{\frac{x_k}{|x|}}
+
\frac{1}{|x|}
\brac{
 \gamma(t)\sin\psi(t)+\delta(t)\cos\psi(t)
}
\frac{x_k}{|x|}\theta}^2  \\[2mm]
=&\abs{\brac{
  \alpha(t)w(t)+\beta(t) v(t) }_2
 \nabla\brac{\frac{x_k}{|x|}}}^2
 +\frac{1}{r^2} \abs{
 \brac{
  \gamma(t)w(t)+\delta(t)v(t)
 }_2
 \frac{x_k}{|x|}}^2
 \end{split}
\]
Thus we obtain for $i=1,2$ and $1\leq k\leq n$,
\[
 \abs{\Omega_{i,k+2}(x)}^2
 =\bigl(\alpha(t)w(t)+\beta(t)v(t)\bigr)_i^2\abs{\nabla\brac{\frac{x_k}{|x|}}}^2
 +\frac{1}{r^2}\bigl(\gamma(t)w(t)+\delta(t)v(t)\bigr)_i^2\brac{\frac{x_k}{|x|}}^2.
\]
Summing over $i=1,2$ and using that $w(t),v(t)$ are orthonormal, so
\[
 \sum_{i=1}^2\bigl(\alpha(t)w(t)+\beta(t)v(t)\bigr)_i^2=\alpha(t)^2+\beta(t)^2,
 \qquad
 \sum_{i=1}^2\bigl(\gamma(t)w(t)+\delta(t)v(t)\bigr)_i^2=\gamma(t)^2+\delta(t)^2,
\]
we get
\[
 \sum_{i=1}^2\abs{\Omega_{i,k+2}(x)}^2
 =\brac{\alpha(t)^2+\beta(t)^2}\abs{\nabla\brac{\frac{x_k}{|x|}}}^2
 +\frac1{r^2}\brac{\gamma(t)^2+\delta(t)^2}\brac{\frac{x_k}{|x|}}^2.
\]
Summing over $1\leq k\leq n$ and using \eqref{eq:general-spherical-identities},
\begin{equation}\label{eq:general-Omega-mixed-norm}
 \sum_{i=1}^2\sum_{k=1}^n\abs{\Omega_{i,k+2}(x)}^2
 =\frac1{r^2}
 \Bigl[
  (n-1)\brac{\alpha(t)^2+\beta(t)^2}
  +\gamma(t)^2+\delta(t)^2
 \Bigr].
\end{equation}
By antisymmetry $\Omega_{k+2,i}=-\Omega_{i,k+2}$, so $\abs{\Omega_{k+2,i}(x)}^2=\abs{\Omega_{i,k+2}(x)}^2$, and the same sum \eqref{eq:general-Omega-mixed-norm} is contributed again by the entries $\Omega_{k+2,i}$, $1\leq k\leq n$, $i=1,2$. Together with
\[
 \abs{\Omega_{12}(x)}^2+\abs{\Omega_{21}(x)}^2=\frac{2\kappa(t)^2}{r^2}
\]
and $\Omega_{ij}=0$ for $i=j$ or $3\leq i,j\leq n+2$, we conclude
\[
 \abs{\Omega(x)}^2
 =\sum_{i,j=1}^{n+2}\abs{\Omega_{ij}(x)}^2
 =\frac{2\kappa(t)^2}{r^2}
 +\frac2{r^2}
 \Bigl[
  (n-1)\brac{\alpha(t)^2+\beta(t)^2}
  +\gamma(t)^2+\delta(t)^2
 \Bigr],
\]
which is \eqref{eq:general-Omega-size}.

Lastly, \eqref{eq:energy-identity} follows by polar coordinates from \eqref{eq:general-gradient-size}
\[
\begin{split}
  \int_{\B^n}\abs{\nabla u}^n\,\dd x  =&\abs{\S^{n-1}} \int_{r=0}^1 \brac{
  f(t)^2
  +\brac{\partial_tg(t)}^2
  +(n-1)g(t)^2
 }^{\frac{n}{2}} \frac{\dd r}{r}\\
 =&\abs{\S^{n-1}} \int_{t=1}^\infty \brac{
  f(t)^2
  +\brac{\partial_tg(t)}^2
  +(n-1)g(t)^2
 }^{\frac{n}{2}} \dd t\\
  \end{split}
\]
We can conclude.
\end{proof}

\section{Proof of Theorem~\ref{thm:Ln-counterexample}}

\begin{proof}
Set
\begin{equation}\label{eq:first-profiles}
 f(t)\coloneqq t^{-1/2},
 \qquad
 \psi(t)\coloneqq 2\sqrt t,
 \qquad
 g(t)\coloneqq t^{-1}.
\end{equation}
and 
\begin{equation}\label{eq:first-coefficients}
 \alpha(t)\coloneqq 0,
 \qquad
 \gamma(t)\coloneqq 0,
 \qquad
 \kappa(t)\coloneqq f(t).
\end{equation}
Then \eqref{eq:general-phase-relation}, \eqref{eq:general-positive-profiles}, and \eqref{eq:general-coefficient-relation} are satisfied.

Let $\delta$ and $\beta$ be defined by
\eqref{eq:general-delta} and \eqref{eq:general-beta}.

Set (as in \eqref{eq:general-mu})
\begin{equation}\label{eq:first-mu}
 \mu(t)\coloneqq
 \brac{
  t^{-1}+(n-1)t^{-2}+t^{-4}
 }^{\frac{n-2}{2}} =  \brac{
  f(t)^2
  +\brac{\partial_tg(t)}^2
  +(n-1)g(t)^2
 }^{\frac{n-2}{2}}.
\end{equation}
Then
\begin{equation}\label{eq:first-log-mu}
 \frac{\partial_t\mu(t)}{\mu(t)}
 =\frac{n-2}{2}
 \frac{-t^{-2}-2(n-1)t^{-3}-4t^{-5}}
      {t^{-1}+(n-1)t^{-2}+t^{-4}}.
\end{equation}
Hence for some dimensional constant $C_n > 0$, it holds
\begin{equation}\label{eq:first-log-mu-bound}
 \forall t\geq 1,\qquad \abs{\frac{\partial_t\mu(t)}{\mu(t)}} \leq\frac{C_n}{t}.
\end{equation}
Since
\[
 \partial_tg=-t^{-2},
 \qquad
 \partial_t^2g=2t^{-3},
\]
\eqref{eq:general-delta}, or rather \eqref{eq:general-deltav2} gives
\begin{equation}\label{eq:first-delta}
 \delta(t)
 =t^{1/2}
 \brac{
  (n-1)t^{-1}
  -2t^{-3}
  +\frac{\partial_t\mu(t)}{\mu(t)}t^{-2}
 }.
\end{equation}
and thus
\begin{equation}\label{eq:first-delta-bound}
 \abs{\delta(t)}\leq C_n t^{-1/2}, \quad t \in [1,\infty).
\end{equation}
Likewise, \eqref{eq:general-beta} or rather \eqref{eq:general-betav2} gives
\begin{equation}\label{eq:first-beta}
 \beta(t)
 =\frac{t}{n-1}
 \brac{
  \frac12t^{-3/2}
  -\frac{\partial_t\mu(t)}{\mu(t)}t^{-1/2}
  -t^{-2}\delta(t)
 }.
\end{equation}
and thus 
\begin{equation}\label{eq:first-beta-bound}
 \abs{\beta(t)}\leq C_n t^{-1/2}, \quad t \in [1,\infty).
\end{equation}

Define $u$ by \eqref{eq:general-u} and define
$\Omega$ by \eqref{eq:general-Omega}.  Explicitly,
\begin{equation}\label{eq:first-u-explicit}
 u(x)=
 \begin{pmatrix}
 \displaystyle \cos\brac{2\sqrt{\log\frac e{\abs{x}}}} \\[5mm]
 \displaystyle \sin\brac{2\sqrt{\log\frac e{\abs{x}}}} \\[5mm]
 \displaystyle \frac{x}{\abs{x}\log(e/\abs{x})}
 \end{pmatrix}
 \qquad
 \text{for }x\in\B^n\setminus\{0\}.
\end{equation}

Clearly $u$ is bounded, but does not have a continuous extension in $0$.

\Cref{lem:common-computation} gives
\begin{equation}\label{eq:first-system-punctured}
 -\div\brac{\abs{\nabla u}^{n-2}\nabla u^i}
 =
 \sum_{j=1}^{n+2}
 \Omega_{ij}\cdot
 \abs{\nabla u}^{n-2}\nabla u^j
 \qquad
 \text{in }\B^n\setminus\{0\}
\end{equation}
for every $i=1,\dots,n+2$.

By \eqref{eq:general-Omega-size},
\[
 \begin{split}
 \abs{\Omega(x)}^2
 =
 \sum_{i,j=1}^{n+2}\abs{\Omega_{ij}(x)}^2=
 \frac2{r^2}
 \brac{
  \kappa(t)^2
  +(n-1)\brac{\alpha(t)^2+\beta(t)^2}
  +\gamma(t)^2
  +\delta(t)^2
 }.
\end{split}
\]
By \eqref{eq:first-coefficients},
\eqref{eq:first-delta-bound}, and
\eqref{eq:first-beta-bound}, and \eqref{eq:general-gradient-size}, we obtain 
\begin{equation}\label{eq:first-Omega-upper-lower}
 |\nabla u(x)| \aeq_{n} \abs{\Omega(x)} \aeq_{n} \frac{1}{|x| \sqrt{\log (e/|x|)}},  
  \quad x \in \B^n \setminus \{0\}.
\end{equation}
Hence we have \eqref{eq:thm1-Omega-size} and by Lemma \ref{lem:Lorentz-profile}, with $a=1/2$, we have \eqref{eq:thm1-full-Lorentz-range}, $\Omega, \nabla u \in L^{(n,q)} (\B^n)$ for any $q> 2$.

In particular, since $n>2$ we have $\nabla u, \Omega \in L^n(\B^n)$. Since $u$ is bounded, we obtain that $u \in W^{1,n} \cap L^\infty(\B^n)$.

It only remains to show that \eqref{eq:first-system-punctured} holds in distributional sense in the whole disk -- which is a standard argument since the right-hand side is $L^1$ and $u \in W^{1,n}$: Let $\varphi \in C_c^\infty(\B^n)$, then for small $\eps > 0$ and $\B_{\eps}$ the ball centred at the origin and of radius $\eps$,
\[
\begin{split}
 \int_{\B^n \setminus \B_{\eps}} |\nabla u|^{n-2} \nabla u \nabla \varphi
 =\int_{\B^n\setminus \B_{\eps} } \Omega \cdot |\nabla u|^{n-2} \nabla u\, \varphi +\int_{\partial \B_{\eps}} |\nabla u|^{n-2} \nabla u\cdot \nu \varphi\\
 \end{split}
 \]
Observe that by \eqref{eq:first-Omega-upper-lower}
\[
\begin{split}
 \abs{\int_{\partial \B_{\eps}} |\nabla u|^{n-2} \nabla u\cdot \nu\, \varphi}
 \aleq\eps^{n-1} \frac{\|\varphi\|_{L^{\infty}(\B)}}{\eps^{n-1} \log^{\frac{n-1}{2}}(e/\eps)} \xrightarrow{\eps \to 0} 0.
 \end{split}
\]
Thus we obtain the distributional equation by taking $\eps \to 0$. This proves \Cref{thm:Ln-counterexample}.
\end{proof}

\section{Proof of Theorem~\ref{thm:gradient-counterexample}}

\begin{proof}
Assume $n\geq4$ and set
\begin{equation}\label{eq:second-profiles}
 f(t)\coloneqq t^{-3/8},
 \qquad
 \psi(t)\coloneqq \frac85t^{5/8},
 \qquad
 g(t)\coloneqq t^{-17/16}.
\end{equation}
Then \eqref{eq:general-phase-relation} and \eqref{eq:general-positive-profiles} are satisfied.

Let $\delta(t)$ and $\beta(t)$ be defined by \eqref{eq:general-delta} and \eqref{eq:general-beta}, and set
\begin{align}
 \alpha(t) &\coloneqq -\frac{\delta(t)+\partial_t\beta(t)}{f(t)}, \label{eq:second-alpha}\\[2mm]
 \gamma(t) &\coloneqq -\partial_t\alpha(t)+\beta(t)f(t), \label{eq:second-gamma}\\[2mm]
 \kappa(t) &\coloneqq \frac{ f(t)^2-(n-1)g(t)\alpha(t)+\partial_tg(t)\gamma(t) }{f(t)}. \label{eq:second-kappa}
\end{align}
Clearly, \eqref{eq:general-coefficient-relation} is satisfied.

Define $u$ by \eqref{eq:general-u}, that is 
\begin{equation}\label{eq:second-u-explicit}
 u(x)=
 \begin{pmatrix}
  \cos\brac{\frac85t^{5/8}}\\[3mm]
  \sin\brac{\frac85t^{5/8}}\\[3mm]
  t^{-17/16}\theta
 \end{pmatrix},
 \qquad
 t=\log\frac e{\abs{x}},
 \quad
 \theta=\frac{x}{\abs{x}}.
\end{equation}
Clearly $u$ is bounded, but does not have a continuous extension in $0$.

Set
\begin{equation}\label{eq:second-zeta}
 \zeta(t)\coloneqq \int_1^t\kappa(s)\,\dd s,
\end{equation}
and, for $1\leq i,j\leq n+2$, define
\begin{equation}\label{eq:second-Xi}
 \Xi_{ij}(x)\coloneqq 
 \begin{cases}
 \zeta(t), &i=1,\ j=2, \\[2mm]
 -\zeta(t), &i=2,\ j=1, \\[2mm]
 \displaystyle \brac{\alpha(t)\cos\psi(t)-\beta(t)\sin\psi(t)}\frac{x_k}{|x|}, &i=1,\ j=k+2,\ 1\leq k\leq n, \\[4mm]
 \displaystyle \brac{\alpha(t)\sin\psi(t)+\beta(t)\cos\psi(t)}\frac{x_k}{|x|}, &i=2,\ j=k+2,\ 1\leq k\leq n, \\[4mm]
 -\Xi_{ji}(x), &i=k+2,\ j\in\{1,2\},\ 1\leq k\leq n, \\[2mm]
 0, &i=j\ \text{or}\ 3\leq i,j\leq n+2.
 \end{cases}
\end{equation}
Set
\begin{equation}\label{eq:second-Omega-definition}
 \Omega\coloneqq \nabla\Xi
 \qquad
 \text{on }\B^n\setminus\{0\}.
\end{equation}

We check that this $\Omega$ agrees with \eqref{eq:general-Omega}. With $w(t)$, $v(t)$ as in \eqref{eq:general-w-v}, set $c(t)\coloneqq \alpha(t)w(t)+\beta(t)v(t)$, so that $c_1(t)=\alpha(t)\cos\psi(t)-\beta(t)\sin\psi(t)$ and $c_2(t)=\alpha(t)\sin\psi(t)+\beta(t)\cos\psi(t)$. By \eqref{eq:general-w-v-derivatives},
\begin{equation}\label{eq:second-c-derivative}
 \partial_tc(t)
 =\brac{\partial_t\alpha(t)-\beta(t)f(t)}w(t)
 +\brac{\alpha(t)f(t)+\partial_t\beta(t)}v(t).
\end{equation}
By \eqref{eq:second-alpha} and \eqref{eq:second-gamma},
\begin{equation}\label{eq:second-c-derivative-relations}
 \partial_t\alpha(t)-\beta(t)f(t)=-\gamma(t),
 \qquad
 \alpha(t)f(t)+\partial_t\beta(t)=-\delta(t),
\end{equation}
so that
\begin{equation}\label{eq:second-c-derivative-final}
 -\partial_tc(t)=\gamma(t)w(t)+\delta(t)v(t).
\end{equation}
Since $\Xi_{i,k+2}(x)=c_i(t)\,x_k/|x|$ and $\nabla t=-\theta/r$,
\[
 \nabla\Xi_{i,k+2}(x)
 =c_i(t)\nabla\brac{\frac{x_k}{|x|}}
 -\frac{\partial_tc_i(t)}r\frac{x_k}{|x|}\theta
 \overset{\eqref{eq:second-c-derivative-final}}=
 c_i(t)\nabla\brac{\frac{x_k}{|x|}}
 +\frac1r\bigl(\gamma(t)w_i(t)+\delta(t)v_i(t)\bigr)\frac{x_k}{|x|}\theta,
 \qquad
 i=1,2,
\]
which is exactly the $(i,k+2)$ entry of \eqref{eq:general-Omega}. 

Since also $\nabla\Xi_{12}(x)=\partial_t \zeta(t)\nabla t=-\kappa(t)\theta/r$, this shows $\nabla\Xi=\Omega$ as in \eqref{eq:general-Omega}. 
Therefore \Cref{lem:common-computation} yields
\begin{equation}\label{eq:second-system-punctured}
 -\div\brac{\abs{\nabla u}^{n-2}\nabla u^i}
 =
 \sum_{j=1}^{n+2}
 \Omega_{ij}\cdot
 \abs{\nabla u}^{n-2}\nabla u^j
 \qquad
 \text{in }\B^n\setminus\{0\}
\end{equation}
for every $i=1,\dots,n+2$.

We now estimate the coefficients. By \eqref{eq:second-profiles},
\begin{equation}\label{eq:second-gradient-profile}
 f(t)^2+\brac{\partial_tg(t)}^2+(n-1)g(t)^2
 =t^{-3/4}\sigma(t),
 \qquad
 \sigma(t)\coloneqq 1+(n-1)t^{-11/8}+\frac{289}{256}t^{-27/8},
\end{equation}
and, since $1\leq\sigma(t)\leq\sigma(1)$ for $t\geq1$, \eqref{eq:general-gradient-size} gives
\begin{equation}\label{eq:second-gradient-size}
 \abs{\nabla u(x)}\aeq_n\frac1{r\,t^{3/8}},
 \qquad
 x\in\B^n\setminus\{0\}.
\end{equation}
In particular $u \in W^{1,n} \cap L^\infty(\B^n)$.

Set
\begin{equation}\label{eq:second-mu}
 \mu(t)\coloneqq \brac{f(t)^2+\brac{\partial_tg(t)}^2+(n-1)g(t)^2}^{\frac{n-2}2}
 =t^{-\frac{3(n-2)}8}\sigma(t)^{\frac{n-2}2},
\end{equation}
and
\begin{equation}\label{eq:second-epsilon}
 \varepsilon(t)\coloneqq \frac{n-2}2\frac{\partial_t\sigma(t)}{\sigma(t)},
 \qquad\text{so that}\qquad
 \frac{\partial_t\mu(t)}{\mu(t)}=-\frac{3(n-2)}{8t}+\varepsilon(t).
\end{equation}
Direct differentiation of \eqref{eq:second-gradient-profile} gives, for some dimensional $C_n>0$,
\begin{equation}\label{eq:second-epsilon-bounds}
 \abs{\partial_t^k\varepsilon(t)}\leq C_nt^{-19/8-k},
 \qquad
 k=0,1,2,
 \qquad
 t\geq1.
\end{equation}

Using $\partial_tg(t)=-\frac{17}{16}t^{-33/16}$ and $\partial_t^2g(t)=\frac{561}{256}t^{-49/16}$, \eqref{eq:general-delta} (or rather \eqref{eq:general-deltav2}) gives the exact identity
\begin{equation}\label{eq:second-delta-formula}
 \delta(t)
 =(n-1)t^{-11/16}
 -\frac{102n+357}{256}t^{-43/16}
 +\frac{17}{16}\varepsilon(t)t^{-27/16}.
\end{equation}
We also have
\begin{equation}\label{eq:general-betav3}
\begin{split}
 \beta(t)
 &\coloneqq \frac1{(n-1)g(t)}
 \Bigg[
 -
 \frac{
 \partial_t\brac{
  \mu(t)
  f(t)
 }
 }{\mu(t)}
 +\partial_tg(t)\delta(t)
 \Bigg].
\end{split}
\end{equation}
and \eqref{eq:second-mu}-\eqref{eq:second-profiles} gives
\begin{equation}\label{eq:second-beta-formula}
 \beta(t)
 =\frac38t^{-5/16}
 -\frac1{n-1}\varepsilon(t)t^{11/16}
 -\frac{17}{16(n-1)}t^{-1}\delta(t).
\end{equation}
By \eqref{eq:second-epsilon-bounds}, \eqref{eq:second-delta-formula} and \eqref{eq:second-beta-formula},
\begin{align}
 \delta(t)&=(n-1)t^{-11/16}+O_n\brac{t^{-43/16}},
 &
 \partial_t\delta(t)&=-\frac{11(n-1)}{16}t^{-27/16}+O_n\brac{t^{-59/16}},
 \label{eq:second-delta-expansion}\\
 \beta(t)&=\frac38t^{-5/16}+O_n\brac{t^{-27/16}},
 &
 \partial_t\beta(t)&=-\frac{15}{128}t^{-21/16}+O_n\brac{t^{-43/16}}.
 \label{eq:second-beta-expansion}
\end{align}
Substituting into \eqref{eq:second-alpha} gives
\begin{equation}\label{eq:second-alpha-expansion}
 \alpha(t)=-(n-1)t^{-5/16}+\frac{15}{128}t^{-15/16}+O_n\brac{t^{-37/16}},
 \qquad
 \partial_t\alpha(t)=\frac{5(n-1)}{16}t^{-21/16}+O_n\brac{t^{-31/16}}.
\end{equation}
Consequently, \eqref{eq:second-gamma} gives
\begin{equation}\label{eq:second-gamma-expansion}
 \gamma(t)=\frac38t^{-11/16}-\frac{5(n-1)}{16}t^{-21/16}+O_n\brac{t^{-31/16}},
\end{equation}
and, by \eqref{eq:second-kappa},
\begin{equation}\label{eq:second-kappa-expansion}
 \kappa(t)=t^{-3/8}+(n-1)^2t^{-1}+O_n\brac{t^{-13/8}}.
\end{equation}
In particular,
\begin{equation}\label{eq:second-coefficient-bounds}
 \abs{\alpha(t)}+\abs{\beta(t)}\leq C_nt^{-5/16},
 \qquad
 \abs{\gamma(t)}+\abs{\delta(t)}\leq C_nt^{-11/16},
 \qquad
 \abs{\kappa(t)}\leq C_nt^{-3/8},
 \qquad
 t\geq1.
\end{equation}

By \eqref{eq:general-Omega-size} and \eqref{eq:second-coefficient-bounds},
\begin{equation}\label{eq:second-Omega-upper}
 \abs{\Omega(x)}\leq\frac{C_n}{r\,t^{5/16}},
 \qquad
 x\in\B^n\setminus\{0\}.
\end{equation}
Since $n\geq4$, $\frac{5n}{16}>1$, so \Cref{lem:Lorentz-profile} with $a=5/16$, $\sigma=n$ gives
\[
 \Omega\in L^n(\B^n).
\]
Exactly as above in the proof of \Cref{thm:Ln-counterexample} we can now argue that \eqref{eq:system} is satisfied in distributional sense and conclude.

\end{proof}

\bibliographystyle{abbrv}
\bibliography{bib}

\end{document}